\documentclass[a4paper,leqno,11pt]{amsart}
\usepackage{txfonts}
\usepackage[T1]{fontenc}
\usepackage[utf8]{inputenc}
\usepackage{graphicx}
\usepackage{mathrsfs}
\usepackage[english]{babel}
\usepackage{amssymb,hyperref,amsmath,amsthm,frcursive}
\usepackage{xcolor}
\swapnumbers
\theoremstyle{plain}
\newtheorem{theo}{Theorem}[section]
\newtheorem{prop}[theo]{Proposition}
\newtheorem{lemm}[theo]{Lemma}
\newtheorem{coro}[theo]{Corollary}

\theoremstyle{definition}
\newtheorem{defi}[theo]{Definition}

\newtheorem*{rem*}{Remark}

\DeclareMathAlphabet{\mathrmsl}{OT1}{cmr}{m}{sl}
\renewcommand{\leq}{\leqslant}
\renewcommand{\geq}{\geqslant}

\newcommand{\NM}{{\mathbb N}}

\newcommand{\RM}{{\mathbb R}}

\newcommand{\DD}{\mathscr{D}}
\newcommand{\BL}{\mathcal{B}}
\newcommand{\DL}{\mathcal{D}}
\newcommand{\BR}{\widetilde{\mathcal{B}}}
\newcommand{\DR}{\widetilde{\mathcal{D}}}

\DeclareMathOperator{\Ric}{Ric}

\newcommand{\vol}{\operatorname{vol}}
\newcommand{\Scal}{\operatorname{Scal}}
\newcommand{\Hess}{\operatorname{Hess}}
\newcommand{\tr}{\operatorname{tr}}

\newcommand{\Id}{\operatorname{Id}}

\newcommand{\even}{\operatorname{even}}
\newcommand{\odd}{\operatorname{odd}}
\newcommand{\parity}{\bullet}

\makeatletter
\newcommand*\owedge{\mathpalette\@owedge\relax}
\newcommand*\@owedge[1]{%
  \mathbin{%
    \ooalign{%
      $#1\m@th\bigcirc$\cr
      \hidewidth$#1\m@th\wedge$\hidewidth\cr
    }%
  }%
}
\makeatother
\newcommand{\kn}{\owedge}
\newenvironment{proofof}[1]{\addvspace{\bigskipamount}\noindent{\em Proof of #1. }}%
{\noindent\hfill $\qed$\par\addvspace{\bigskipamount}}
\newenvironment{sketchofproofof}[1]{\addvspace{\bigskipamount}\noindent{\em Sketch of proof of #1. }}%
{\noindent\hfill $\qed$\par\addvspace{\bigskipamount}}
\newcounter{mnotecount}[section]
\renewcommand{\themnotecount}{\thesection.\arabic{mnotecount}}
\newcommand{\mnote}[1]
{\protect{\stepcounter{mnotecount}}$^{\mbox{\footnotesize  $
      \bullet$\themnotecount}}$ \marginpar{\raggedright\tiny\em
    $\!\!\!\!\!\!\,\bullet$\themnotecount: #1} }

\swapnumbers

\makeatletter
\@addtoreset{equation}{section}
\makeatother

\renewcommand{\footnote}[1]{${}^($\footnotemark${}^)$ \footnotetext{#1}}

\begin{document}
\title[Gauss-Bonnet-Chern center of mass]{The Gauss-Bonnet-Chern center of mass\\ for asymptotically flat manifolds}
\author{Marc Herzlich} 
\address{Institut montpelli\'erain Alexander Grothendieck\\ Universit\'e de Montpellier\\ CNRS (France)} 
\email{marc.herzlich@umontpellier.fr}
\date{\today}
\thanks{Part of this work was done during a visit at the Mittag-Leffler Institute in Stockholm during the 2019 fall semester \emph{General Relativity, Geometry and Analysis: beyond the first 100 years after Einstein}, supported by the Swedish Research Council under grant no.\! 2016-06596. The author is most grateful to the Institute for its ideal conditions for research.}
\thanks{The author is supported in part by the project \emph{Curvature constraints and moduli spaces of metrics} of the French National Agency for Research under grant no.\! ANR-17-CE40-0034.}
\begin{abstract} 
In this paper we introduce a family of center of masses that complement the definition of the family of Gauss-Bonnet-Chern masses by Ge-Wang-Wu and Li-Nguyen. In order to prove the existence and the well-definedness of the center of mass, we use the formalism of double forms of Kulkarni and Labbi. This allows for transparent conceptual proofs, which apply to all known cases of asymptotic invariants of asymptotically flat manifolds.
\end{abstract}

\maketitle

\section*{Introduction}

Asymptotically flat manifolds have been widely studied for more than 30 years. This class of Riemannian manifolds were initially introduced in General Relativity, where they are used as spacelike slices in spacetimes modeling isolated systems such as the gravitational field generated by an isolated star or a black hole, but they also appeared in a number of other contexts, such as, \emph{e.g.}, singular limits of sequences of Riemannian manifolds.

As emphasized by physics, a special feature of asymptotically flat manifolds is the existence of peculiar geometrically well-defined asymptotic invariants depending on the first derivative of the metric in a neighbourhood of inifinity. The most prominant example of these is the so-called ADM mass of asymptotically flat manifolds, first defined Arnowitt, Deser, and Misner \cite{adm}. From the viewpoint of physics, ADM mass is a measure of the total amount of energy in spacetime. Another important invariant is the center of mass, which stands as the relativistic analogue of the center of gravity of the Newtonian theory of gravity. 

The above invariants exist when the difference between the asymptotically flat metric at hand and the background Euclidean metric belongs to a very specific regime, measured by the speed of convergence towards zero at infinity, see below for details. Much more recently, Ge, Wang and Wu \cite{gww-gbc-mass}, and independently by Li and Nguyen \cite{nguyen-li}, defined a family of asymptotic invariants, similar to the ADM mass, but that make sense for slower speeds of decay. They called them \emph{Gauss-Bonnet-Chern masses}, as they are related to the curvature polynomials appearing in the Gauss-Bonnet-Chern(-Weil) formulas for the Euler characteristic in even dimensions, a feature that is similar to the link between the classical ADM mass and the total scalar curvature which is the integrand for the Gauss-Bonnet formula in dimension $2$.

The goal of the present paper is two-fold. In a first part, we shall pursue further this work, by exhibiting a new Gauss-Bonnet-Chern \emph{center of mass}. With the notable exception of \cite{wang-wu-chern-magic}, where the possibility of existence of a center of mass is briefly discussed (without any proof nor an explicit expression), its existence seems to have escaped notice so far. The second goal of the paper is to advertise for the use of the formalism of double forms to handle computations related to this type of asymptotic invariants. The formalism of double forms was initially introduced by Kulkarni \cite{kulkarni-bianchi}, and it was then thoroughly developped in the recent years by Labbi, see \cite{labbi-double-forms,labbi-variational} for instance. In our specific context, it enables us to write the proof of the existence and well-definedness of the invariants in a conceptual way, which is much more efficient than the usual computations in coordinates (compare with \cite[p.94--97]{gww-gbc-mass} which only considers the lowest degree invariant in the family). Relying on it, we shall then not only provide a proof of the existence and well-defined character of the new center of mass, but we shall also show how this applies to the invariants previously defined by Ge, Wang, and Wu and Nguyen and Li. 

The paper is organized as follows. In a first section, we shall describe the geometric context at hand. Section \ref{sec:double-forms} provides the background on the formalism of double forms which is systematically used in the paper. This will allow in Section \ref{sec:geo-invariance-GBC-mass} for a renewed definition of the Gauss-Bonnet-Chern masses due to Ge, Wang, and Wu and Li and Nguyen, and the geometric well-definedness of these mass invariants will be re-proven in Section \ref{sec:geo-invariance-GBC-mass} using the formalism of double forms. Section \ref{sec:geo-invariance-GBC-center-mass} is then devoted to our main result, stated in Theorem \ref{th:existence-GBCcenter}: the precise definition of a new invariant, the Gauss-Bonnet-Chern center of mass. We the proceed to prove its existence and geometric invariance. Final remarks are collected in a concluding section, including a first-derivative-free definition of the center of mass, following ideas of \cite{herzlich-curvature-mass,wang-wu-chern-magic}. 

\bigskip

{\flushleft\textbf{Acknowledgements}}. The author is grateful to Gautier Dietrich for the numerous discussions we had together on higher-order masses, which played a role in the birth of the current paper. Yuxin Ge is also to be thanked for several discussions about his work with his collaborators.

\section{Motivation and definitions}

We shall be working on asymptotically flat manifolds, defined as follows.

\begin{defi}[Bartnik \cite{bartnik-mass-af}, Lee-Parker \cite{lp}]\label{defi:af-mfds}
An asymptotically flat manifold of $C^{\ell}$-regularity ($\ell\geq 2$) is a complete Riemannian manifold $(M,g)$ such that there exists a diffeomorphism $\Phi$ 
(called a chart at infinity) from the complement of a compact set in $M$ into the complement of a ball in $\RM^n$, 
such that, in these coordinates and for some $\tau>0$, called the \emph{order of decay},
\begin{equation}\label{eqn:af-metric}
g_{ij} - \delta_{ij} = O\left(|x|^{-\tau}\right), \quad 
\partial_k g_{ij} = O\left(|x|^{-\tau-1}\right), \quad 
\dots ,\quad 
\partial_{k_1}\!\!\cdots\partial_{k_\ell} g_{ij} = O\left(|x|^{-\tau-\ell}\right),
\end{equation} 
where $|x|$ is the Euclidean radius in $\RM^n$.
\end{defi}

The usual definition of asymptotically flat manifols only considers two orders of differentiability ($C^\ell$-regularity with $\ell=2$), but we shall need $\ell=3$ in the main part of this paper. As shown by Bando, Kasue, and Nakajima \cite{bkn}, the definition of asymptotic flatness is geometric as it can also be defined equivalently by volume growth and curvature decay conditions. We have however preferred Definition \ref{defi:af-mfds} for its simple formulation. An important fact is the following rigidity result for asymptotic isometries, independently due to Bartnik \cite{bartnik-mass-af} and Chru\'sciel \cite{ptc-erice}.

\begin{theo}[Bartnik \cite{bartnik-mass-af}, Chrusciel \cite{ptc-erice}]\label{theo:asymptotic-rigidity-of isometries}
Let $g$ be asymptotically flat metric of $C^\ell$-regularity ($\ell\geq 2$) and decay order $\tau>0$ on $\RM^n\setminus \bar{B}$ and $\Phi$ a diffeormophism of $\RM^n\setminus\bar{B}$ such that
$$ (\Phi^*g)_{ij} - g_{ij} = O\left(|x|^{-\tau}\right), \quad 
\dots, \quad 
\partial_{k_1}\!\!\cdots\partial_{k_\ell}\left((\Phi^*g)_{ij} - g_{ij}\right) = O\left(|x|^{-\tau-\ell}\right).$$
Then there exists $r\!\gg\!1$, an (affine) isometry $A$ of the Euclidean space, and a vector field $\zeta$ on $\RM^n\setminus\bar{B}_r$ such that $\Phi = A \circ S$, where $S$ is the diffeomorphism of $\RM^n\setminus\bar{B}_r$ defined by $S(x) = x + \zeta(x)$, and for any $\tau'<\tau$,
$$ 
\zeta^j = O\left(|x|^{-\tau'+1}\right), \quad 
\partial_k\zeta^i = O\left(|x|^{-\tau'}\right), \quad 
\dots, \quad 
\partial_{k_1}\!\!\cdots\partial_{k_\ell}\zeta^i = O\left(|x|^{-\tau'-(\ell-1)}\right).$$
\end{theo}

Theorem \ref{theo:asymptotic-rigidity-of isometries} gives sense to the notion of \emph{the} limiting Euclidean space of an asymptotically flat manifold, as any two such spaces differ at infinity by the action of an Euclidean isometry. More precisely, given two charts at infinity
$$ \Phi_i : M\setminus K  \rightarrow \RM^n\setminus\bar{B} \qquad (i=1,2),$$
where the push-forward metrics $(\Phi_i^{-1})^*g = g_i = b+e_i$ on $\RM^n\setminus\bar{B}$ (where $b$ refers to the Euclidean metric in the given chart at infinity) satisfy conditions \eqref{eqn:af-metric} of Definition \ref{defi:adm-mass}. Thus, 
$$(\Phi_1\circ\Phi_2^{-1})^*g_1 - g_1 =  g_2 - g_1 = (g_2 - b) - (g_1 -b) = e_2 - e_1 $$
satisfies the assumptions of Theorem \ref{theo:asymptotic-rigidity-of isometries} and we conclude that
$$ \Phi_1\circ\Phi_2^{-1} = A \circ S$$
where $A$ and $S$ are defined as in that Theorem. At infinity, the change of charts thus reduces to $A$ and there is a well defined and unique limiting Euclidean space, independent of the choice of chart.

These preliminaries being said, we can now define the main invariant of asymptotically flat manifolds, originally defined by physicists in the context of mathematical relativity:
  
\begin{defi}[Arnowitt, Deser, and Misner \cite{adm}]\label{defi:adm-mass} 
Let $(M,g)$ be a ($C^2$-regular) asymptotically flat manifold of order $\tau$. If $\tau>\tfrac{n-2}{2}$ and the scalar curvature of $g$ is integrable, the quantity
\begin{equation}\label{eqn:defi:adm-mass}
m(g) = \lim_{r\to\infty} \int_{S_r} \left(-d^*g - d \tr g\right) (\nu_r)\,d\!\vol^b_{S_r} 
\end{equation}
exists (where  $d^*$ is the Euclidean divergence defined as the adjoint of the exterior derivative, the trace is taken with respect to the Euclidean metric, and $\nu_r$ denotes the field of Euclidean outer unit normals to the coordinate spheres $S_r$) and is independent of the chart chosen around infinity. It is called the \emph{mass} 
of the asymptotically flat manifold $(M,g)$.
\end{defi}

It is obvious that the mass vanishes if $\tau>n-2$ and it has been shown that the condition $\tau>\tfrac{n-2}{2}$ is necessary for it to be well-defined independently of the chart chosen to compute it. Thus, existence and well-definedness of the mass holds when the order of decay belongs to the interval
$$ \frac{n-2}{2} \ < \ \tau \ \leq \ n-2.$$
Note that in coordinates, the integrand in Formula \eqref{eqn:defi:adm-mass} can be rewritten as
\begin{equation}\label{eqn:mass-avec-deltas}
\left( \partial_i g_{ij} - \partial_j g_{kk} \right) \nu_r^j 
= \partial_{\ell} g_{jk}\, \left( \delta^{k\ell}\delta_i^j - \delta^{jk}\delta_i^{\ell} \right)\, \nu_r^i 
\end{equation}
(where the Einstein summation convention holds on pairs of repeated indices), an expression which we will meet again later. Note also that we do not use here the usual normalization of the invariants: the values given in Definitions \ref{defi:adm-mass} as well as in Definition \ref{defi:center-mass-adm} below differ from the usual ones by constants depending only on the dimension.

One of the many motivations for this definition stems from some Hamiltonian analysis of the Einstein equations seen as a dynamical system on the space of asymptotically flat metrics. The dynamical system is invariant under symmetries of the model metric at infinity, \emph{i.~e.}~the Minkowski metric, and the mass is an admissible expression for the Hamiltonian of the dynamical system subject to the constraints. It is thus a conserved quantity under the evolution equations, associated to the time-translation $\tfrac{\partial}{\partial t}$ relative to the choice of a flat spacelike Euclidean space $\mathbb{R}^n$ in Minkowski space.

We now proceed to the definition of the (classical) center of mass of general relativity, following the approach initially due to Regge and Teitelboim \cite{regge-teitelboim}, see for instance Huang \cite{huang-center-mass,huang-center-mass-general} for rigorous proofs:

\begin{defi}[Regge and Teitelboim \cite{regge-teitelboim}]\label{defi:RT-conditions} 
An asymptotically flat manifold $(M,g)$ satisfies the $C^\ell$-regular ($\ell\geq 2$) \emph{Regge-Teitelboim} (RT) conditions of order $\tau>0$  if it is  \emph{asymptotically flat of order} $\tau$ and
$$ 
g_{ij}^{\odd} = O\left(|x|^{-\tau-1}\right), \quad
\left(\partial_kg_{ij}\right)^{\even} = O\left(|x|^{-\tau-2}\right), \quad 
\dots , \quad
\left(\partial_{k_1}\cdots\partial_{k_\ell}g_{ij}\right)^{\parity} = O\left(|x|^{-\tau-\ell-1}\right),$$ 
where $\cdot^{\odd}$ and $\cdot^{\even}$ denote the odd and even parts of a function on the chosen chart at infinity and $\parity$ is odd (resp.~even) if $\ell$ is even (resp.~odd). For sake of simplicity, such a metric will be said to be \emph{RT-asymptotically flat}.
\end{defi}

From the discussion above, the following rigidity statement is natural, and its proof can be obtained \emph{mutatis mutandis} from the original proof.

\begin{theo}[\emph{\`a la Chru\'sciel}]\label{theo:RT-asymptotic-rigidity}
Let $(M,g)$ be a $C^\ell$-regular ($\ell\geq 2$) RT-asymptotically flat manifold of decay order $\tau>0$ on $\RM^n\setminus \bar{B}$ and $\Phi$ a diffeormophism of $\RM^n\setminus\bar{B}$ such that
$$ 
(\Phi^*g)_{ij} - g_{ij} = O\left(|x|^{-\tau}\right), \ 
\quad \dots, \ 
\quad \partial_{k_1}\!\!\cdots\partial_{k_\ell}\left((\Phi^*g)_{ij} - g_{ij}\right) = O\left(|x|^{-\tau-\ell}\right),
$$
and 
$$ 
\left((\Phi^*g)_{ij} - g_{ij}\right)^{\odd} = O\left(|x|^{-\tau-1}\right), \quad 
\dots , \quad 
\left(\partial_{k_1}\!\!\cdots\partial_{k_\ell}\left((\Phi^*g)_{ij} - g_{ij})\right)\right)^{\parity} = O\left(|x|^{-\tau-\ell-1}\right),
$$
where $\parity$ is odd (resp.~even) if $\ell$ is even (resp.~odd). Then there exists $r\!\gg\!1$, an (affine) isometry $A$ of the Euclidean space, and a vector field $\zeta$ on $\RM^n\setminus\bar{B}_r$ such that $\Phi = A \circ S$, where $S$ is the diffeomorphism of $\RM^n\setminus\bar{B}_r$ defined by $S(x) = x + \zeta(x)$, and for any $\tau'<\tau$,
$$ 
\zeta^i = O\left(|x|^{-\tau'+1}\right), \quad 
\dots , \quad
\partial_{k_1}\!\!\cdots\partial_{k_\ell}\zeta^i = O\left(|x|^{-\tau'-\ell+1}\right),
$$
and
$$ 
\left(\partial_j\zeta^i + \partial_i\zeta^j \right)^{\odd} = O\left(|x|^{-\tau'-1}\right), \quad
\dots , \quad
\left(\partial_{k_1}\!\!\cdots\partial_{k_{\ell-1}}\left(\partial_j\zeta^i + \partial_i\zeta^j\right)\right)^{\parity} = O\left(|x|^{-\tau'-\ell}\right)
$$
where $\parity$ is odd (resp.~even) if $\ell$ is even (resp.~odd).
\end{theo}

\begin{sketchofproofof}{Theorem \ref{theo:RT-asymptotic-rigidity}}
The first part of the estimates are the result of Theorem \ref{theo:asymptotic-rigidity-of isometries}. It thus remains to obtain the last set of estimates, but these are somehow obvious since $\mathcal{L}_{\zeta}b$, whose expression in coordinates is just $\partial_k\zeta^i + \partial_i\zeta^k$, is the highest order part in the difference between $\Phi^*b$ and $b$, hence between $\Phi^*g$ and $g$.
\end{sketchofproofof}

The significance of the parity conditions introduced in the Regge-Teitelboim conditions will be explained after the definition of the center of mass, which we now state.

\begin{defi}[Regge and Teitelboim \cite{regge-teitelboim}, Huang \cite{huang-center-mass,huang-center-mass-general}]\label{defi:center-mass-adm}
If the metric is RT-asymptotically flat, $\tau>\tfrac{n-2}{2}$, the scalar curvature $\Scal^g$ of $g$ is integrable,  
$\left(\Scal^g\right)^{\odd} = O\left(r^{-2\tau-2}\right)$, and $m(g)\neq 0$,
the quantity
\begin{equation*}\label{eqn:defi:center-mass-adm}
c^{i}(g) \ = \ \frac{1}{m(g)}\, \lim_{r\to\infty} \int_{S_r} 
\bigl[ \, x^{i}(-\delta^b g - d \tr_b g) 
- (g-b)(\partial_i,\cdot) + \tr_b (g-b)\, dx^{i}\,\bigr] (\nu)\,d\!\vol^b_{s_r} 
\end{equation*}
exists for each $i$ in $\{1,...,n\}$. Moreover, the point $\mathbf{C}(g) = (c^1(g),\dots,c^n(g))$ in $\mathbb{R}^n$ is independent 
of the chart chosen around infinity, up to the action of rigid Euclidean isometries. It is called the \emph{ADM center of mass} of the asymptotically flat manifold $(M,g)$.
\end{defi}

The center of mass must be understood as a point in the (unique) limiting Euclidean space which is identified to some $\mathbb{R}^n$ in each chart at infinity: its expression in coordinates is thus naturally \emph{equivariant} under the action of the changes of coordinates at infinity (the transformations $A$) from Theorem \ref{theo:RT-asymptotic-rigidity}, and not \emph{invariant} as is the mass (which is a number). In the above Hamiltonian description, the center of mass is related to the invariance of the dynamical system under boosts, \emph{i.e.} Lorentz rotations. Whereas translations are unambiguously defined in Minkowski space, specifying a subgroup of boosts (which are linear transformations) involves the choice of an origin, and the RT conditions are just a convenient way to underline the special part played by the origin in each chart at infinity , as a mean to choose a specific subgroup of boosts. 

Much more recently, Ge, Wang, and Wu  \cite{gww-gbc-mass}, and independently Nguyen and Li  \cite{nguyen-li}, discovered that the ADM mass belongs to a larger family of invariants $m_k(g)$, which were called \emph{Gauss-Bonnet-Chern masses} in reference to the integrands in the Gauss-Bonnet formulas. The definition of the higher order invariants indeed the complete contractions of the $k$-th powers of the curvature operator of the metric $g$, or equivalently, the integrands of the Gauss-Bonnet-Chern formulas. The precise definitions of these quantities, called $L_k^g$ below, will be given later in Sections \ref{sec:double-forms} and \ref{sec:geo-invariance-GBC-mass}, but we may already state that $L_1^g$ is the scalar curvature in dimension $n\geq 2$ whereas $L_2^g$ is a multiple of $|R^g|^2 - 4|\Ric^g|^2 + (\Scal^g)^2$ in dimension $n\geq 4$. Their most important feature is that one can always write $L_k^g$ as a product
\begin{equation}\label{eqn:Lk=RkPk}
L_k^g \ = \ (R^g)^{mnrs}\, \left( P_k^g\right)_{mnrs}
\end{equation}  
where $P_k^g$ is a polynomial of degree $k-1$ in the curvature tensor $R^g$ of the metric. For example, if $k=1$,
$$ (P_1^g)_{mnrs} \ = \ \delta_{rs}\delta_{mn} - \delta_{nr}\delta_{ms},$$
a quantity that we have already met in Formula \eqref{eqn:mass-avec-deltas}.

\begin{defi}[Ge, Wang, and Wu \cite{gww-gbc-mass}, Nguyen and Li \cite{nguyen-li}]\label{defi:gbc-mass} 
If $n\geq 2k+1$, $\tau>\tfrac{n-2k}{k+1}$, and $L_k^g$ is integrable, the quantity
\begin{equation}\label{eqn:defi:gbc-mass}
m_k(g) = a_{n,k} \ \lim_{r\to\infty} \int_{S_r} \partial_{\ell} g_{jk}\, 
\left(\, (P_k^g)_{mnrs}\,\delta^{jn}\delta^{kr}\delta^{\ell s}\delta_i^m \,\right)\nu^i\,d\!\vol^e_{s_r} 
\end{equation}
exists and is independent of the chart chosen around infinity. It is called the \emph{Gauss-Bonnet-Chern mass} of the asymptotically flat manifold $(M,g)$.
\end{defi}

Here $a_{n,k}$ is a dimensional constant depending only on $n$ and $k$, chosen so that the $k$-th order mass of the generalized (Riemannian) Schwarzschild metric
$$ g_{S,k} \ = \ \left( 1 \, +\, \frac{m}{2r^{\frac{n}{k}-2}}\right)^{\frac{4k}{n-2k}} \, b $$
is precisely equal to $m^k$.

From the form of $P_1^g$ given above, one sees that $m_1(g)$ and the ADM mass $m(g)$ of Definition \ref{defi:adm-mass} coincide. Each of these generalized masses
$m_k(g)$ is well-defined and possibly non-zero on asymptotically flat manifolds of dimensions $n\geq 2k+1$ whose order of decay belongs to the following range:
$$ \frac{n-2k}{k+1} < \tau \leq \frac{n-2k}{k}.$$
Thus the whole family of Gauss-Bonnet-Chern masses (including the first one) covers a large range of possible decays of asymptotically flat manifolds: $0 < \tau \leq n-2$. Note that an analogue has also been introduced in the asymptotically hyperbolic setting where the part of the Euclidean space is played by the hyperbolic space \cite{chrusciel-herzlich-hyperbolic-mass,gww-hyperbolic}. They are defined with similar formulas but they won't play any role in our paper and we shall not give their precise definitions. It is however clear that the techniques developed in the present paper (the use of double forms for the proofs of existence of the asymptotic invariants) apply to this case as well.

The proofs of the statements included in Definitions \ref{defi:adm-mass} and \ref{defi:gbc-mass} (that the invariants are well-defined geometric invariants) proceed as follows. In a first step, one easily checks that the invariants behave nicely under isometries of the Euclidean space. Then the rigidity theorems \ref{theo:asymptotic-rigidity-of isometries} and \ref{theo:RT-asymptotic-rigidity} yield formulas for the differences between the invariants computed in two different charts at infinity, involving the vector field $\zeta$. In the last step (which is the main one), it is shown that this difference is the sum of a term that decays rapidly to zero at infinity and an extra term that doesn't decay fast enough to vanish at infinity. However, the extra term has a divergence structure, so that it disappears as each invariant is the result of an integration over spheres of radii tending to infinity. 

The main difficulty in these proofs is to perform the computations in the last step to make the divergence structure apparent, and we shall show in the next sections that the formalism of double forms provides an efficient tool for this. We also note that Ge, Wang, and Wu did not give a complete proof of the geometric well-definedness of their invariants for any value $k$ in \cite{gww-gbc-mass}, as they restricted themselves to the case where $k=2$ for the proof; our approach thus gives a unified proof for all higher-order masses. Following this, we will proceed in turn to the definition of our new center of mass in Section \ref{sec:geo-invariance-GBC-center-mass}, together with the proof that it is well-defined under the relevant Regge-Teitelboim conditions.

\section{The formalism of double forms}\label{sec:double-forms}

The formalism of double forms has been introduced by Kulkarni to manage more easily computations with the Bianchi identities and Bochner-type formulas. We shall follow the extension considered by Labbi in the papers \cite{labbi-double-forms,labbi-variational,labbi-bianchi-pontrjagin}, which are our main references for what follows.

\begin{defi} A double form on a manifold $M$ is an element of $\DD^{p,q} = \Lambda^pT^*M \otimes \Lambda^qT^*M$.
The total space of double forms is $\DD = \bigotimes\limits_{p,q} \DD^{p,q}$.
\end{defi}

In the following, we fix a Riemannian metric on the manifold $M$, which will be supposed to be oriented. The main algebraic features of double forms are collected in the following statements (note that we correct in (\ref{item:etoile-produit-scalaire}) below a computational mistake in \cite[Proposition 3.2]{labbi-double-forms}).

\begin{prop}[Labbi \cite{labbi-double-forms}]\label{prop:double-forms-algebra}
The following structures are available on $\DD$:
\vspace{-10pt}
\begin{enumerate}
\item It is an algebra for the Kulkarni-Nomizu product $\kn$ defined on simple elements by 
$$ (\omega_1\otimes\omega_2)\kn (\varphi_1\otimes\varphi_2) = (\omega_1\wedge\varphi_1)\otimes (\omega_2\wedge\varphi_2).$$
\item The contraction map $c :  \DD^{p,q} \rightarrow  \DD^{p-1,q-1}$ is the adjoint of the Kulkarni-Nomizu multiplication (on whichever side) by the metric $g$ for the natural (tensor-product) scalar product on double forms:
$$ \langle g\owedge\omega , \varphi\rangle = \langle \omega , c\varphi\rangle .$$
\item the Hodge star operator is the operator induced by the usual Hodge duality:
$$ * (\omega_1\otimes\omega_2) = (*\omega_1)\otimes(*\omega_2) .$$
Thus, $*^2 = (-1)^{(p+q)(n-p-q)}\Id$ on $\DD^{p,q}$.
\item\label{item:etoile-produit-scalaire} For any pair of double forms $\omega$ and $\varphi$ in $\DD^{p,q}$, 
$$ \langle \omega , \varphi\rangle = *\, ( \omega \owedge\! * \varphi ) = (-1)^{(p+q)(n-p-q)} * (*\, \omega \owedge \varphi ).$$
\item For any $k\leq n$, $*\, g^{\kn k} = \tfrac{k!}{(n-k)!}\, g^{\kn (n-k)}$.
\end{enumerate} 
\end{prop}

For notational convenience, we shall from now on denote any orthonormal basis by $\{e_k\}_{1\leq k\leq n}$, and we shall also freely identify forms and vectors.

\begin{defi}[Labbi \cite{labbi-double-forms}]
The left (or first) Bianchi map $\BL : \DD^{p,q} \rightarrow \DD^{p+1,q-1}$ is defined by
$$ \BL (\omega ) = - \sum_{k=1}^n e_k \kn \iota_{\widetilde{e}_k}\omega $$
where $\iota$ is the contraction between (double) forms and (double) vectors \emph{w. r. t.} the metric.
The right (or second) Bianchi map $\BR : \DD^{p,q} \rightarrow \DD^{p-1,q+1}$ is defined by
$$ \BR (\omega ) = - \sum_{k=1}^n \iota_{e_k}\omega\kn\widetilde{e}_k .$$
\end{defi}

We also note that there is a natural \emph{transpose} map from $\DD^{p,q}$ to $\DD^{q,p}$, and the two Bianchi operators are conjugate through transposition. Moreover, it is easy to get the following.

\begin{lemm}[Labbi \cite{labbi-bianchi-pontrjagin}]
The Bianchi maps are anti-derivations: for any pair of elements $\omega$ and $\varphi$ of $\DD^{p,q}$ and $\DD^{r,s}$,
$$\BL(\omega\kn\varphi) = \BL\omega\kn\varphi + (-1)^{p+q}\BL\omega\kn\BL\varphi , \, \textrm{ and } \, \BR(\omega\kn\varphi) = \BR\omega\kn\varphi + (-1)^{p+q}\BR\omega\kn\BL\varphi .$$
\end{lemm}

The Bianchi operators enable us to recover the usual first Bianchi identity for the curvature of a Riemannian metric, which can be written as $\BL(R^g)=\BR(R^g)=0$ (here both left and right Bianchi maps apply since the Riemannian curvature is a symmetric double form).

\begin{defi}[Labbi \cite{labbi-variational}]
The algebra if double forms admits two natural differential operators:
\vspace{-10pt}
\begin{enumerate}
\item the left exterior derivative $\DL : \DD^{p,q}\rightarrow \DD^{p+1,q} $ defined by
$$ \DL\omega = - \sum_{k=1}^n e_k \kn \nabla^g_{e_k}\omega $$
where $\nabla^g$ is the Levi-Civita connection, and
\item the right exterior derivative $\DR : \DD^{p,q}\rightarrow \DD^{p+1,q} $ defined by
$$ \DR\omega = - \sum_{k=1}^n \nabla^g_{e_k}\omega \kn \widetilde{e}_k  .$$
\end{enumerate}
\end{defi} 

Obviously, the exterior derivatives are conjugate through transposition. One also notes that $\DL\omega= -d\omega$ if $\omega$ is a $(p,0)$-form, and similarly $\DR\varphi = - d\varphi$ on $(0,q)$-forms. Moreover, the following Leibniz rule is easily proven:

\begin{prop}\label{prop:Leibniz}
For elements $\omega$ and $\varphi$ of $\DD^{p,q}$ and $\DD^{r,s}$,
\vspace{-10pt}
\begin{itemize}
\item[(1)] $\DL(\omega\kn\varphi) = \DL\omega \kn \varphi + (-1)^p \omega \owedge \DL\varphi$,
\item[(2)] $\DR(\omega\kn\varphi) = \DR\omega \kn \varphi + (-1)^q \omega \owedge \DR\varphi$.
\end{itemize}
\end{prop}

Another well-known fact is that both $\DL^2$ and $\DR^2$ are curvature terms, hence are of $0$th-order. The same also holds for $[\DL,\DR] = \DL\DR - \DR\DL$, as a simple computation shows. A nice observation, which seems to have escaped notice so far, is that the exterior derivatives also have simple commutation properties with the Bianchi maps. This feature will moreover play a central role in the proofs to follow.

\begin{prop}\label{prop:commutations}
In the algebra of double forms, 
\vspace{-10pt}
\begin{itemize}
\item[(1)] $ \BL\DL = - \DL\BL$ and $\BR\DR = - \DR\BR$,
\item[(2)] $ \BR\DL = - \DL\BR - \DR$ and $\BL\DR = - \DR\BL - \DL$.
\end{itemize}
\end{prop}

Note that we use the convention here that $\BL=0$ on $(p,0)$-forms and $\BR=0$ on $(0,q)$-forms. 

\begin{proof}
We first check that $\BL\DL = - \DL\BL$. Let $\omega\in\DD^{p,q}$: from the definition of the exterior derivative and the Bianchi map, and computing in an orthonormal basis that is parallel at the point of interest,
$$
\BL\DL\omega \ = \ - \, \sum_{k=1}^n e_k \kn \iota_{\tilde{e}_k} (\DL\omega) \ = \ \sum_{k=1}^n\sum_{\ell=1}^n e_k \kn \iota_{\tilde{e}_k}\left( e_{\ell}\kn \nabla_{e_{\ell}}\omega \right) \ = \ \sum_{k=1}^n\sum_{\ell=1}^n e_k \kn e_{\ell}\kn \iota_{\tilde{e}_k} \nabla_{e_{\ell}}\omega
$$
whereas 
\begin{equation*}
\begin{split}
\DL\BL\omega \ 
& = \ - \, \sum_{k=1}^n e_k \kn \nabla_{e_{k}}(\BL\omega) \\
& = \ \sum_{k=1}^n\sum_{\ell=1}^n e_k \kn \nabla_{e_{k}}\left( e_{\ell}\kn \iota_{\tilde{e}_{\ell}}\omega \right) \\
& = \ \sum_{k=1}^n\sum_{\ell=1}^n e_k \kn e_{\ell} \kn  \iota_{\tilde{e}_{\ell}}\nabla_{e_{k}}\omega \\
& = \ \sum_{k=1}^n\sum_{\ell=1}^n e_{\ell} \kn e_k \kn  \iota_{\tilde{e}_{k}}\nabla_{e_{\ell}}\omega \\
& = \ - \, \sum_{k=1}^n\sum_{\ell=1}^n  e_k \kn e_{\ell} \kn  \iota_{\tilde{e}_{k}}\nabla_{e_{\ell}}\omega \ = \ - \, \BL\DL\omega .
\end{split}
\end{equation*}
By transposition, one gets that $\BR\DR = - \DR\BR$.

We now manage the second equality, which requires a little bit of work. For any $\omega\in\DD^{p,q}$, it is easy to check that, for
$$ \BR(\DL \omega) \ = \ \sum_{k=1}^n\sum_{\ell=1}^n \tilde{e}_{\ell}\kn \iota_{e_\ell}\left( e_k\wedge \nabla_{e_k}\omega\right).$$
We now compute, for $(x_1,\dots, x_{p+1},y_1,\dots,y_q)$ in $(TM)^{p+q+1}$,
\begin{equation*}
\begin{split}
 \sum_{k=1}^n \iota_{x_1}\left( e_k\wedge \nabla_{e_k}\omega\right) & (x_2\wedge\dots\wedge x_{p+1};y_1\wedge\dots\wedge y_q) \\
 & \ = \ \sum_{k=1}^n  \left( e_k\wedge \nabla_{e_k}\omega\right) (x_1\wedge x_2\wedge\dots\wedge x_{p+1};y_1\wedge\dots\wedge y_q) \\
 & \ = \ \sum_{k=1}^n \frac{1}{p!} \sum_{\sigma\in\mathfrak{S}_{p+1}} \varepsilon(\sigma) \, e_k(x_{\sigma(1)})\nabla_{e_k}\omega \, (x_{\sigma(2)}\wedge\dots\wedge x_{\sigma(p+1)};y_1\wedge\dots\wedge y_q) \\
 & \ = \ \frac{1}{p!} \sum_{j=1}^{p+1} \ \sum_{\substack{\sigma\in\mathfrak{S}_{p+1}\\ \sigma(1)=j}} \ \sum_{k=1}^n \varepsilon(\sigma)\,  e_k(x_{\sigma(1)})\nabla_{e_k}\omega \, (x_{\sigma(2)}\wedge\dots\wedge x_{\sigma(p+1)};y_1\wedge\dots\wedge y_q) \\
 & \ = \ \frac{1}{p!} \sum_{j=1}^{p+1} \ \sum_{\substack{\sigma\in\mathfrak{S}_{p+1}\\ \sigma(1)=j}} \, \varepsilon(\sigma)\,  \nabla_{x_j}\omega \,  (x_{\sigma(2)}\wedge\dots\wedge x_{\sigma(p+1)};y_1\wedge\dots\wedge y_q).
\end{split}
\end{equation*}
A permutation such that $\sigma(1)=j$ can be written as a product $\sigma = \tau_1\rho\tau_2$ where $\tau_1$, resp.~$\tau_2$, is the transposition exchanging $j$ and $p+1$, resp.~$1$ and $p+1$, and $\rho$ belongs to $\mathfrak{S}_p$.
Hence, 
\begin{equation*}
\begin{split}
 \sum_{k=1}^n \iota_{x_1}\left( e_k\wedge \nabla_{e_k}\omega\right) & (x_2\wedge\dots\wedge x_{p+1};y_1\wedge\dots\wedge y_q) \\
 & \ = \ \frac{1}{p!} \sum_{j=1}^{p+1} \ \sum_{\rho\in\mathfrak{S}_{p}} \, \varepsilon(\tau_1) \varepsilon(\tau_2)\varepsilon(\rho)\,  \nabla_{x_j}\omega \, (x_{\tau_1\rho\tau_2(2)}\wedge\dots\wedge x_{\tau_1\rho\tau_2(p+1)};y_1\wedge\dots\wedge y_q) \\
  & \ = \ \frac{1}{p!} \sum_{j=1}^{p+1} \ \sum_{\rho\in\mathfrak{S}_{p}} \, \varepsilon(\tau_1) \varepsilon(\tau_2)\varepsilon(\rho)\,  \nabla_{x_j}\omega \, (x_{\tau_1\rho(2)}\wedge\dots\wedge x_{\tau_1\rho(p)}\wedge x_{\tau_1\rho(1)};y_1\wedge\dots\wedge y_q)\\
  & \ = \ \frac{(-1)^{p+1}}{p!} \sum_{j=1}^{p+1} \ \sum_{\rho\in\mathfrak{S}_{p}} \, \varepsilon(\tau_1) \varepsilon(\tau_2)\varepsilon(\rho)\,  \nabla_{x_j}\omega \, (x_{\tau_1\rho(1)}\wedge\dots\wedge x_{\tau_1\rho(p)};y_1\wedge\dots\wedge y_q)\\
  & \ = \ (-1)^{p+1} \sum_{j=1}^{p+1} \, \varepsilon(\tau_1) \varepsilon(\tau_2)\,  \nabla_{x_j}\omega \, (x_{\tau_1(1)}\wedge\dots\wedge x_{\tau_1(p)};y_1\wedge\dots\wedge y_q)\\
  & \ = \ (-1)^{p+1} \sum_{j=1}^{p} \, \varepsilon(\tau_1) \varepsilon(\tau_2)\,  \nabla_{x_j}\omega \, (x_1\wedge\dots\wedge x_{p+1}\wedge\dots\wedge x_p;y_1\wedge\dots\wedge y_q) \\
  & \ \ \ \ \ \ \ \ \  \ \ \ \ \ \ \ \ \ + \ (-1)^{p+1}\varepsilon(\tau_1) \varepsilon(\tau_2)\,  \nabla_{x_{p+1}}\omega \, (x_1\wedge\dots\wedge x_p;y_1\wedge\dots\wedge y_q),
\end{split}
\end{equation*}
where in the last but one line the subscript $p+1$ is in the $j$-th position. Thus,
\begin{equation*}
\begin{split}
 \sum_{k=1}^n \iota_{x_1}\left( e_k\wedge \nabla_{e_k}\omega\right) & (x_2\wedge\dots\wedge x_{p+1};y_1\wedge\dots\wedge y_q) \\
 & \ = \ \sum_{j=1}^{p} \, (-1)^{j+1}  \varepsilon(\tau_1) \varepsilon(\tau_2)\,  \nabla_{x_j}\omega \, (x_1\wedge\dots\wedge \widehat{x_j}\wedge x_{j+1}\wedge\dots\wedge x_{p+1};y_1\wedge\dots\wedge y_q) \\
 & \ \ \ \ \ \ \ \ \ \ \ \ \ \ \ \ \ \ + \ (-1)^{p+1}\varepsilon(\tau_1) \varepsilon(\tau_2)\,  \nabla_{x_{p+1}}\omega \, (x_1\wedge\dots\wedge x_p;y_1\wedge\dots\wedge y_q)\\
 & \ = \ \sum_{j=1}^{p} \, (-1)^{j+1} \,  \nabla_{x_j}\omega \, (x_1\wedge\dots\wedge \widehat{x_j}\wedge x_{j+1}\wedge\dots\wedge x_{p+1};y_1\wedge\dots\wedge y_q) \\
 & \ \ \ \ \ \ \ \ \ \ \ \ \ \ \ \ \ \ + \ (-1)^{p}\,  \nabla_{x_{p+1}}\omega \, (x_1\wedge\dots\wedge x_p;y_1\wedge\dots\wedge y_q)\\
 & \ = \ - \ \sum_{j=1}^{p+1} \, (-1)^{j} \,  \nabla_{x_j}\omega \, (x_1\wedge\dots\wedge \widehat{x_j}\wedge\dots\wedge x_{p+1};y_1\wedge\dots\wedge y_q) \\
  & \ = \ - \ \sum_{j=2}^{p+1} \, (-1)^{j} \,  \nabla_{x_j}\omega \, (x_1\wedge\dots\wedge \widehat{x_j}\wedge\dots\wedge x_{p+1};y_1\wedge\dots\wedge y_q) \\
  & \ \ \ \ \ \ \ \ \ \ \ \ \ \ \ \ \ \ + \  \nabla_{x_1}\omega \, (x_2\wedge\dots\wedge x_{p+1};y_1\wedge\dots\wedge y_q)\\[7pt]
  & \ = \ \left(\DL(\iota_{x_1}\omega) + \nabla_{x_1}\omega\right)\, (x_2\wedge\dots\wedge x_{p+1};y_1\wedge\dots\wedge y_q)
\end{split}
\end{equation*}
We thus conclude that $
\BR\left(\DL \omega\right) 
\ = \ \sum_{\ell=1}^n \, \tilde{e}_{\ell}\kn  \left(\DL(\iota_{e_{\ell}}\omega) + \nabla_{e_{\ell}}\omega\right)  
\ = \ - \, \DR\omega \ - \ \DL\left( \BL\omega \right)
$, and by transposition, one gets that $\BL\DR = - \DR\BL - \DL$ as well.
\end{proof}

The last ingredient that will be needed in the proof are the adjoints of the exterior derivatives.

\begin{defi}[Labbi \cite{labbi-variational}]
The adjoints of the exterior derivatives are
$$ \DL^* = (-1)^{n(p+1) + q(n-q)}*\DL\,* \ \ \textrm{ and } \ \ \DR^* = (-1)^{n(q+1) + p(n-p)}*\DR *\ .$$
The operators $\delta = - \DL^*$ and $\tilde{\delta} = - \DR^*$ are the \emph{left and right divergences} on the algebra of double forms.
\end{defi}

These formulas reproduce the usual relations between the divergence, the Hodge star operator, and the adjoint of the exterior derivative on differential forms (note that we correct here another computational mistake in \cite[Propositions 3.1 and 3.2]{labbi-variational}). 

The formalism of double forms provides an efficient way to define the curvature invariants that are the building blocks of the higher order invariants. 

\begin{defi}\label{defi:Lk}
For $k\leq \tfrac{n}{2}$, the $k$-th Gauss-Bonnet-Chern curvature of the metric $g$ is the following polynomial in the curvature tensor 
$$ L_k = \frac{1}{(n-2k)!}\, * \left(R^{\kn k} \kn g^{\kn (n-2k)} \right) ,$$
\end{defi}

From Proposition \ref{prop:double-forms-algebra}, it is clear that $L_k$ is an element of $\DD^{0,0}$. Using properties (2-5) of the same Proposition \ref{prop:double-forms-algebra}, it is for instance easy to check  that $L_1 = \Scal^g$ (as expected) while $L_2$ is a multiple of $|R^g|^2 - 4|\Ric^g|^2 + (\Scal^g)^2$. 

Before proceeding further, one needs to make a last point clear. There are one-to-one projection maps $\pi$ and $\widetilde{\pi}$ from $\DD^{p,0}$ and $\DD^{0,p}$ to $\Omega^p$ which yield the same image for elements that are transpose of each other in the algebra of double forms. In the sequel, we shall thus freely interpret forms in $\Omega^p$ as elements of either $\DD^{p,0}$ or $\DD^{0,p}$, depending on the context. To make the distinction between the two variants of double forms, we shall decorate elements of $\DD^{0,p}$ with a \emph{tilde}. For instance, a vector $X$ or a form $\theta$, when seen as elements of $\DD^{1,0}$, will be denoted by the same notation $X$ or $\theta$, whereas their tilded versions $\widetilde{X}$ and $\widetilde{\theta}$ denote the corresponding elements of $\DD^{0,1}$.

\section{The Chern-Gauss-Bonnet mass and center of mass of asymptotically flat manifolds}\label{sec:geo-invariance-GBC-mass}

Our goal in this section is to give the definition of the Gauss-Bonnet-Chern masses and our new center of mass using the formalism of double forms, and to derive the proof of existence and well-definedness of the Gauss-Bonnet-Chern mass.  Although this section brings nothing new compared to \cite{gww-gbc-mass}, we hope to convince the reader that our approach makes the computations much easier to manage. The knowledge gained in this section will be used again in Section~\ref{sec:geo-invariance-GBC-center-mass} where the case of the center of mass will be considered.

We first begin with a restatement of Formula \eqref{eqn:Lk=RkPk} in the formalism of double forms. Its proof is immediate from Definition  \ref{defi:Lk}, Proposition \ref{prop:double-forms-algebra}, and the Leibniz rule of Proposition \ref{prop:Leibniz}.

\begin{lemm}
For a Riemannian metric $g$ with curvature $R^g$, let $P_k^g$ be the element of $\DD^{n-2,n-2}$ defined by 
$$ *P_k^g = \frac{1}{(n-2k)!}\, \left(R^g\right)^{\kn(k-1)} \kn g^{\kn(n-2k)}.$$
Then $L_k = *\left( R^g \kn\! *\! P_k\right)$. Moreover, $\DL^g \left(*P_k^g\right)$, $\DR^g\left(*P_k^g\right)$,  $\BL^g \left(*P_k^g\right)$, and $\BR^g\left(*P_k^g\right)$ vanish.
\end{lemm}

We can now state the main properties of the Gauss-Bonnet-Chern mass of Ge-Wang-Wu and Li-Nguyen, within the double forms framework, The definition of our new center of mass is deferred to the next section.

\begin{theo}\label{th:existence-GBCmass}
Let $k\in\NM^*$ and $(M,g)$ be a ($C^3$-regular) asymptotically flat manifold of dimension $n\geq 2k+1$ and order $\tau > \tfrac{n-2k}{k+1}$ , and let $b$ stand for the Euclidean background metric in the chart, so that $g=b+e$.
If the $k$-th Gauss-Bonnet-Chern curvature $L_k$ is in $L^1$, the \emph{Gauss-Bonnet-Chern mass}
$$ m_k(g) \ = \ \frac{(-1)^n}{2(n-1)!\,\omega_{n-1}}\,\lim_{r\to\infty}\, \int_{S_r}\,  * \left( \DR e \kn R^{\kn(k-1)} \kn b^{\kn(n-2k)}\right)(\nu_r)  \ d\!\vol_{S_r}^b $$
is well defined and its value does not depend of the chart at infinity. 
\end{theo}

Here $\omega_{n-1}$ is the volume of the unit radius Euclidean sphere, $\nu_r$ denotes the outer unit normal to the Euclidean sphere of radius $r$ and the Hodge $*$, the product $\kn$, and the derivative $\DR$ are taken with respect to the Euclidean metric $b$, whereas the curvature is $R = R^g$. As in the original definition, the dimensional constant in front of the limit is chosen so that the result of the computation equals $m^k$ for the generalized Schwarzschild (Riemannian) metric.

The difference
$e$ between the metric and the Euclidean background metric is considered as an element of $\DD^{1,1}$, so that the quantity between parentheses in the definition of $m_k(g)$ is an element of $\DD^{n-1,n}$ and its image by the Hodge $*$ is an element of $\DD^{1,0}$, which is identified with a $1$-form before contracting with the vector $\nu_r$. Note that
an alternative expression is
$$ m_k(g) \ = \ \frac{(-1)^n}{2(n-1)!\,\omega_{n-1}}\, \lim_{r\to\infty}\, \int_{S_r}\, \left\langle  \DR e \kn R^{\kn(k-1)} \kn b^{\kn(n-2k)} \, , \, \widetilde{d\!\vol^b} \right\rangle .$$
The proofs of Theorem \ref{th:existence-GBCmass} rely on a number of computations, part of which is common to the case of the center of mass to be studied later.

We begin by recalling the variations of the curvature tensor under changes of metrics. The result which follows is not new but we shall need both a specific expression for the first-order term which we copy from a computation done by Labbi in~\cite{labbi-variational} and a slightly more precise knowledge of the higher-order terms as what is usually needed. For the statement of the latter part, we introduce a little bit of notation. Given an asymptotically flat metric $g$ (of $C^\ell$-regularity with $\ell\geq 2$) and decay order $\tau>0$ and a symmetric bilinear form $h$ such that $g+h$ also is an asymptotically flat metric with the same regularity and decay order, and given a function $\mathsf{A}(h, g, g^{-1})$ of the coefficients of $h$, $g$, and its inverse in a coordinate chart at infinity, we say 
that
$$ \mathsf{A} \ = \ \sum_{m\geq m_0} [h]^m \, + \, O\left(|x|^{-\infty}\right)$$
if there is a formal power series starting at degree $m_0$,
$$ \sum_{m\geq m_0} \ \sum_{i_1\cdots i_mj_1\cdots j_m} \ a^{i_1\cdots i_mj_1\cdots j_m}(g) \, h_{i_1j_1}\cdots h_{i_mj_m}$$
with coefficients depending only on the coefficients of $g$ and its inverse, and satisfying the following: for any $\sigma>0$, there exists $M\geq m_0$ such that
$$ \mathsf{A}(h, g, g^{-1}) \ - \ \sum_{m=m_0}^M \ \sum_{i_1\cdots i_mj_1\cdots j_m} a^{i_1\cdots i_mj_1\cdots j_m}(g) \, h_{i_1j_1}\cdots h_{i_mj_m} \ = \ O\left(|x|^{-\sigma}\right),$$
and the derivatives up to the $\ell$-th one also satisfy similar estimates.
 
\begin{prop} \label{prop:variation-courbure}
Let $g$ a $C^\ell$-regular (with $\ell\geq 2$) Riemannian asymptotically flat metric of order $\tau$ and $h$ a symmetric bilinear form such that $g+h$ also is an asymptotically flat metric with the same regularity and decay order. Then
$$ R^{g+h} = R^g - \frac{1}{4}\left( \DL\DR + \DR\DL \right) h + F(R^g,h) + r_g(h),$$
where  $F$ is a bilinear map with coefficients depending on $g$ and $g^{-1}$ (without derivatives), and, in any chart at infinity,
\begin{equation*}
\begin{split}
r_g(h)  \ = & \ \ L_1(\partial^2 h)\,  \mathsf{A}_1(h, g, g^{-1}) \ + \ L_2(\partial^2 g) \, \mathsf{A}_2(h, g, g^{-1}) \\
& \ \ \ \ \ \ + Q_1(\partial h, \partial h)\,  \mathsf{B}_1(h, g, g^{-1}) \ + \ Q_2(\partial g, \partial g) \, \mathsf{B}_2(h, g, g^{-1})
\end{split}
\end{equation*}
with $L_1$ and $L_2$ linear and $Q_1$ and $Q_2$ bilinear in their arguments with coefficients depending only on $g$ and $g^{-1}$ (again without derivatives), and for $s=1$ or $2$,
$$ \mathsf{A}_s \ = \ \sum_{m\geq s} [h]^m \, + \, O\left(|x|^{-\infty}\right) \ \textrm{ and } \  \mathsf{B}_s \ = \ \sum_{m\geq s} [h]^m \, + \, O\left(|x|^{-\infty}\right) .$$
\end{prop}

\begin{proof} 
Straightforward computation in coordinates, with an identification of the linear term with the one found in \cite[Lemma 4.1]{labbi-variational}.
\end{proof}

Our first key computation is then the following.

\begin{lemm}\label{lem:Lk-est-divergence}
Let $g$ be an asymptotically flat metric, written in a chart at infinity as $g = b+e$ where $b$ is the Euclidean background metric of the chart. Then, for any $k\geq 1$ and any function $V$,
\begin{equation*}
2 \, V \, L_k \  = \  (-1)^{n-1}\, \delta * \left( V \DR e \kn\! *\! P_k \ - \  \DR V\kn e \kn\! *\! P_k \right)  \ + \ *\, \Big( \Hess^b V \kn e \kn\! *\! P_k \Big) \ + \ \mathcal{Q}
\end{equation*}
where $\DL$, $\DR$, $\delta$, and $\tilde{\delta}$ denote the exterior derivatives on double forms and their adjoints for the Euclidean metric $b$, $\Hess^b V$ is the Euclidean Hessian of $V$ seen here as an element of $\DD^{1,1}$, and the remainder $\mathcal{Q}$ satisfies the following pointwise estimate:
\begin{equation*}
 |\mathcal{Q}| \leq 
 C_1 |V|\, |R^g|^{k-2} \left( |e|\, |\nabla^b\nabla^be||R^g|  + |\nabla^be|^2 |R^g| + |e|\, |\nabla^be|\,|\nabla^bR^g| + |r_b(e)||R^g| \right)
 + \ C_2\, |dV|\,|e|\, |\nabla^be|\,|R^g|^{k-1} .
\end{equation*}
\end{lemm}

\begin{proof}
Proposition \ref{prop:variation-courbure} asserts that
\begin{equation*}
-4 L_k \ = \ -4 *_g \!( R^g \kn\! *_g\! P_k^g)  \ = \ *_g\!\left(  (\DR\DL e  + \DL\DR e ) \kn\! *_g\! P_k\right) \ - \ 4 *_g\!\Big( r_b(e) \kn\! *_g\! P_k\Big)  ,
\end{equation*}
where we have used here that $R^b =0$, and we recall that $\DL$ and $\DR$ are the Euclidean exterior derivatives. Replacing the first Hodge star operator $*_g$ by its Euclidean analogue $*$, we get 
\begin{equation*}
-4 L_k \ = \ * \left(  (\DR\DL e  + \DL\DR e ) \kn\! *_g\!P_k\right) \ + \ \mathcal{Q}_1,
\end{equation*}
where $\mathcal{Q}_1 =  \Big(\!*_g-*\!\Big)\!\left(  (\DR\DL e  + \DL\DR e ) \kn\! *_g\! P_k\right)- \ 4 *_g\!\Big( r_b(e) \kn\! *_g\! P_k\Big)$.

Using the formalism of double forms, one first computes
\begin{equation*}
\begin{split}
V\!*\!\left( \DL\DR e \kn\! *_g\! P_k\right) & = *V \DL\left( \DR e \kn\! *_g\! P_k\right)  -  * V \left( \DR e \kn\! \DL(*_gP_k) \right) \\
& =* \left( \DL \left(V\,\DR e\kn\! *_g\! P_k\right)- \DL V \kn\DR e\kn\! *_g\! P_k \right) - * V \left( \DR e \kn\! \DL(*_gP_k) \right)\\
& = (-1)^{n-1} \Big(\!*\!\DL\!*\!\Big)\!*\!(V \DR e\kn\! *_g\! P_k ) - *\!\left( \DL V \kn\DR e\kn\! *_g\! P_k \right) - * V\!\left( \DR e \kn\! \DL(*_gP_k)\right)  \\ 
& = (-1)^{n}\, \delta \left( * (V \DR e\kn\! *_g\! P_k ) \right) - * \left( \DL V \kn\DR e\kn\! *_g\! P_k \right)   -  * V \left( \DR e \kn\! \DL(*_gP_k) \right).
\end{split} 
\end{equation*}
We now use that $\DL^g(* P_k)=0$ to replace the last term by a term involving $\DL-\DL^g$. We also replace all occurences of $*_g$ by $*$ and collect the differences between the expressions for $g$ and those for $b$ in a new remainder term. Thus,
$$  V*\!\left( \DL\DR e \kn\! *\! P_k\right) \ = \ (-1)^{n}\delta\left(*(V \DR e\kn\! *\! P_k ) \right) - * \left( \DL V \kn\DR e\kn\! *\! P_k \right) \ + \ \mathcal{Q}_2.$$
Since $\DL V\kn\DR e = \DR (\DL V\kn e) - \DR\DL V\kn e$ and $\DR\DL V = \DL\DR V = \Hess V$ ($=\Hess^e V$), one gets 
\begin{equation*}
V\!*\!\left( \DL\DR e \kn\! *\! P_k\right) = (-1)^{n}\left[\, \delta\left( *(V\,\DR e\kn\! *\! P_k )\right) - \tilde{\delta}\,\Big( \DL V\kn e \kn\! *\! P_k\Big) \,\right] + * \,\Big(\Hess V\kn e\kn\! *\! P_k\Big) 
+ \mathcal{Q}_3.
\end{equation*}
Here the fact that $\DR(*_gP_k)=0$ has been used again and the same argument as above has taken care of the differences between the terms related to $b$ and those related to $g$, collecting them in the remainder term. 

It now remains to compute $V\!*\!\left( \DR\DL e \kn\! *_g\! P_k\right)$. But this is the transpose of the previous one, thus one obtains a similar result where tilded and untilded operators are just exchanged.  Summing up all contributions and collecting all remainder terms, we end up with
\begin{equation*}
\begin{split}
4 \, V \, L_k \  = & \ (-1)^{n-1}\, \left[ \delta * ( V \DR e \kn\! *\! P_k -   \DR V\kn e \kn\! *\! P_k ) 
+ \tilde{\delta} * ( V \DL e \kn\! *\! P_k  -  \DL V\kn e \kn\! *\! P_k ) \right] \\
& \ \ \ + 2* \left( \Hess V \kn e \kn\! *\! P_k \right)  + \mathcal{Q}
\end{split}
\end{equation*}
We now notice that the two terms within the brackets in the first line are elements of $\DD^{0,0}$ and are transpose to each other, so that they must be equal. This yields the expression given in the statement of the Lemma (changing $\mathcal{Q}$ into $2\mathcal{Q}$).

We now check the estimate on $\mathcal{Q}$, using that the difference between $*_g$ and $*$ is of the same order as $e$ whereas the Levi-Civita connections of $g$ and $b$ differ by a term which involves $\nabla^be$. A careful but easy bookkeeping then yields the expected estimate, and this concludes the proof of the Lemma.
\end{proof}

We can now prove existence of the Gauss-Bonnet-Chern mass. The key idea is to use Lemma \ref{lem:Lk-est-divergence} to relate the integral over spheres appearing in the definition of $m_k(g)$ to integrals over the bulk. We start with a simple estimate. 

\begin{coro}
If $V\equiv 1$ and under the assumptions of Definition \ref{defi:af-mfds}, $\mathcal{Q} = O\left(|x|^{-(k+1)\tau -2k}\right)$.
\end{coro}

We now show how this can be used to give a proof of the existence and geometric invariance of the Chern-Gauss-Bonnet masses.

\vspace{-.4cm}

\begin{proofof}{Theorem~\ref{th:existence-GBCmass}}
We apply Lemma \ref{lem:Lk-est-divergence} with the constant function $V\equiv 1$. Thus, on a domain $D_r$ whose boundary is a coordinate sphere $S_r$ in the chosen chart at infinity, 
$$ c_{n,k} \int_{D_r}  L_k \, d\!\vol^g = \int_{S_r} \left\langle * \left( \DR e \kn R^{\kn(k-1)} \kn b^{\kn(n-2k)} \right), \nu_r \right\rangle \, d\!\vol_{S_r}^b + \int_{D_r} \mathcal{Q} ,$$
where $c_{n,k}$ is a constant involving only $n$ and $k$.

The term in the l.h.s.~converges as $r$ tends to $\infty$ since the $k$-th Gauss-Bonnet-Chern curvature polynomial is in $L^1$, and the second term in the r.h.s.~also converges: since $\tau > \tfrac{n-2k}{k+1}$, one gets that 
$$-(k+1)\tau-2k\ < \ -n ,$$ 
and $\mathcal{Q}$ is in $L^1$  indeed. Hence the Gauss-Bonnet-Chern mass $m_k(g)$ exists. 

It remains to prove that the value of the limit does not depend on the choice of a chart at infinity. We consider two charts
$$  \Phi_1 : M \setminus K   \longrightarrow \mathbb{R}^n\setminus  K_1, \quad \Phi_2 : M \setminus K   \longrightarrow \mathbb{R}^n\setminus  K_2, $$
with $(\Phi_i^{-1})^*g = g_i = b+e_i$, where $b$ is the Euclidean metric in $\mathbb{R}^n$ and $e_i$ for $i=1,2$ satisfy the asymptotic behaviour as in Theorem \ref{th:existence-GBCmass}. Letting $\Phi = \Phi_1 \circ \Phi_2^{-1}$, 
$\Phi^*g - g = g_2 - g_1 = e_2-e_1$ is $O\left(|x|^{-\tau}\right)$ and the control extends up to the third derivative. 
From the asymptotic rigidity of charts at infinity (Theorem \ref{theo:asymptotic-rigidity-of isometries}), we get that $\Phi = A \circ S$ with $A$ some Euclidean isometry. Since the value of the mass is clearly invariant by Euclidean isometries, proving that the mass is independent of the choice of charts is equivalent to proving that 
$$ I_1(r)  =  \int_{S_r}\, \left\langle * \left( \DR e_1 \kn R_1^{\kn(k-1)} \kn b^{\kn(n-2k)}\right), \nu_r \right\rangle \, d\!\vol_{S_r}^b$$
and 
$$ I_2(r) = \int_{S_r}\, \left\langle * \left( \DR e_2 \kn R_2^{\kn(k-1)} \kn b^{\kn(n-2k)}\right), \nu_r \right\rangle \, d\!\vol_{S_r}^b$$
have the same limit as $r$ tends to infinity in the case $e_2-e_1 = \mathcal{L}_{\zeta}b + q$  with the decay conditions on the remainder $q$ given in Theorem \ref{theo:asymptotic-rigidity-of isometries}. (As the reader already guessed, the $R_i$'s  for $i=1,2$ are the curvatures of the asymptotically flat metrics $g_i$.) 

To compute further, we recall that
$$ R_2 = R_1 -\frac{1}{4}\left(\DR\DL +\DL\DR\right)\mathcal{L}_{\zeta}b + F(R_1,\mathcal{L}_{\zeta}b) + r_{b_1+e_1}(e_2-e_1).$$
Hence, there are three different types of terms in the difference between the integrands of $I_1(r)$ and $I_1(r)$,  computed with respect to $g_1$ or to $g_2$:
\vspace{-12pt}
\begin{itemize}
\item the first type comprises only terms containg a product of $\DR(e_2-e_1)$ with a product of $k-1$ curvature terms and $b$ to the power $n-2k$,
\item the second type is formed by terms involving a $\DR e_i$ together with a $(\DR\DL +\DL\DR)\mathcal{L}_{\zeta}b$ and a product of $k-2$ curvatures, and
\item the third type collects all remaining terms, \emph{i.e.} those involving $q$ at least once, or $F(R_1,\mathcal{L}_{\zeta}b)$ at least once, or $r_{b_1+e_1}(e_2-e_1)$ at least once.
\end{itemize}   
\vspace{-11pt}
The main idea behind the proof now is the following: we will get rid of terms of the second and third types by decay considerations, whereas terms of the first type need a more careful study. 

Checking the asymptotic decay assumptions, all the terms of the third type are $O\left(|x|^{-(k+1)\tau-2k+1}\right)$ at least at infinity. As above, the assumption on $\tau$ implies that $-(k+1)\tau-2k+1<-(n-1)$, so that these terms do not contribute in the limit $r\to\infty$ as we integrate over spheres of volume $\sim\!\! r^{n-1}$. 

We then consider the second type, \emph{i.e.} products of $\DR e_i$ with $(\DR\DL +\DL\DR)\mathcal{L}_{\zeta}b$ and $(k-2)$ curvatures. At first glance this may seem to have a bad behaviour since it doesn't decay fast enough to vanish at infinity, but this is an illusion. Following the convention already described, in Section \ref{sec:double-forms}, we denote by $\theta$ the $1$-form dual to $\zeta$ and by $\theta$ and $\tilde{\theta}$ the corresponding elements of $\DD^{1,0}$ and $\DD^{0,1}$, so that $\mathcal{L}_{\zeta}b = - \DL\tilde{\theta} - \DR\theta$. Hence, 
\begin{equation*}
\begin{split}
\left(\DR\DL +\DL\DR\right)\mathcal{L}_{\zeta}b & = - \left(\DR\DL +\DL\DR\right)\left( \DL\tilde{\theta} + \DR\theta\right) \\
& = - \DR\DL\DL\,\tilde{\theta}  - \DL\DR\DL\,\tilde{\theta} - \DR\DL\DR\,\theta - \DL\DR\DR\,\theta \\
& = - \DR\DL^2\,\tilde{\theta}  - \DL\DR^2\,\theta - \DL^2\DR\,\tilde{\theta} - \DR^2\DL\,\theta - \DL[\DL,\DR]\,\tilde{\theta} - \DR[\DL,\DR]\,\theta .
\end{split}
\end{equation*}
We now recall that $\DL^2$, $\DR^2$, and $[\DL,\DR]$ are curvature terms. As a result, we get that $(\DR\DL +\DL\DR)\mathcal{L}_{\zeta}b$ vanishes and this doesn't contribute at infinity as well!

Thus the only terms that might contribute are of the first type, \emph{i.e.} integrals of
$$ \left\langle * \left( \DR(\mathcal{L}_{\zeta}b) \kn R_1^{\kn(k-1)} \kn b^{\kn(n-2k)}\right) , \nu_r \right\rangle .$$
Letting $T = R_1^{\kn(k-1)} \kn b^{\kn(n-2k)}$, we now compute the left hand side in the above inner product:
\begin{equation*}
\begin{split}
*\left( \DR(\mathcal{L}_{\zeta}b) \kn T \right) 
& = - *\left( \DR (\DL\tilde{\theta} + \DR\theta)  \kn T \right) \\
& = - *\left( \DR\DL\tilde{\theta} + \DR\DR\theta  \kn T \right) \\
& = - *\left( \DR\DL\tilde{\theta}  \kn T \right) , \\
\end{split}
\end{equation*}
where we have used once again that $\DR\DR=0$. We now let the commutation properties between the exterior derivatives and the Bianchi map (Proposition~\ref{prop:commutations}) enter the picture:
$$ \mathcal{B}\left(\DR \tilde{\theta}\right) = -  \DR\left( \widetilde{\mathcal{B}}\tilde{\theta}\right) -\DL\tilde{\theta} .$$
Since $ \widetilde{\mathcal{B}}\tilde{\theta} = - \theta$, we get that $\mathcal{B}\left(\DR \tilde{\theta}\right) - \DR\theta=-\DL\tilde{\theta}$ and this yields 
\begin{equation*}
\begin{split}
- *\left( \DR\DL\tilde{\theta}  \kn T \right)
& =  *\left( \DR \left(  \mathcal{B}\left(\DR \tilde{\theta}\right) - \DR\theta \right) \kn T \right) \\
& =  *\left( \,\left(\DR\mathcal{B}\left(\DR \tilde{\theta}\right) - \DR\DR\theta \right) \kn T \right) \\
& =  *\left( \DR\mathcal{B}\left(\DR \tilde{\theta}\right) \kn T \right)
\end{split}
\end{equation*}
We then commute the Bianchi map with the exterior derivative once again, so that
\begin{equation*}
\begin{split}
* \left( \DR\mathcal{B}\left(\DR \tilde{\theta}\right) \kn T \right)
& = *\left(\, \left(- \mathcal{B}\DR - \DL\right) \DR \tilde{\theta} \kn T \right)\\
& = - *\left(\,\left(\mathcal{B}\DR\DR \tilde{\theta} + \DL \DR \tilde{\theta}\right) \kn T \right)\\
& = - *\left(\DL \DR \tilde{\theta} \kn T \right)\\
& = - *\DL\left( \DR \tilde{\theta} \kn T \right)  + *\left( \DR \tilde{\theta} \kn \DL T \right) 
\end{split}
\end{equation*}
Since $\DL^{g_1} R_1 = 0$ (second Bianchi identity), the last term in the r.h.s.~is bounded by $C |\nabla^b\zeta| |\nabla^be||R_1|^{k-1}$, thus it is $O\left(|x|^{-(k+1)\tau-2k+1}\right)$. When $\tau > \tfrac{n-2k}{k+1}$, one has again $-(k+1)\tau-2k+1<-(n-1)$ and the integral of this term over larger and larger coordinate spheres don't contribute at infinity. Defining a $(n-2)$-form $\beta$ by
$$\beta \otimes \widetilde{d\!\vol^b} = \DR \tilde{\theta} \kn T$$
in $\DD^{n-2,n}$, ur computation then implies that
$$ (I_2-I_1)(r) = - \int_{S_r} \left\langle * \DL \left(\DR \tilde{\theta} \kn R_1^{\kn(k-1)} \kn b^{\kn(n-2k)}\right) , \nu_r\right\rangle  \, d\!\vol_{S_r}^b  \, + \, \, o(1) \, = \, \int_{S_r} d \beta   \, + \, o(1) \,  =  \, o(1) ,$$
and this ends the proof.
\end{proofof}

As already explained in the introduction, this proof of the existence and well-definedness of the Gauss-Bonnet-Chern mass is not really different than the original ones by Ge-Wang-Wu or Nguyen-Li. Its interest lies in the fact that it is well suited for generalizations as it gives a clear Ariadne thread for the computations. Indeed, a careful re-reading of the previous two pages shows that there are only two key elements: the identity $\mathcal{L}_{\zeta}b = - \DL\tilde{\theta} - \DR\theta$ and the commutation properties of the Bianchi maps and the exterior derivatives.

\section{Geometric invariance of the Gauss-Bonnet-Chern center of mass}\label{sec:geo-invariance-GBC-center-mass}

The first step in defining a center of mass in the Gauss-Bonnet-Chern setting is to write a precise definition for the asymptotic behaviour of the metric in charts at infinity, \emph{i.e.} the relevant Regge-Teitelboim asymptotic conditions.
Definition \ref{defi:RT-conditions} in the introduction has the disadvantage of relying on the use of coordinates, where coefficients of tensors are understood as functions. Since differentiating an odd (resp.~even) function yields an even (resp.~odd) function, hence one has to keep in mind the order of differentiability. We shall rather use an equivalent definition for tensors.

\begin{defi}
In a chart at infinity, a double form $\omega$ is even if $\sigma^*\omega = \omega$ and odd if $\sigma^*\omega = -\omega$, where $\sigma$ is the antipodal map $x\mapsto -x$ in the chart. 
\end{defi}

It is moreover obvious that a product of two even or two odd double forms is even, whereas a product of an even and an odd double form is odd.  Moreover, it is clear that any Euclidean covariant derivative in a chart of an even (resp.~odd) tensor field is even (resp.~odd). Of course, the \emph{even part} of $\omega$ is 
$$\omega^{\even} = \tfrac12 (\omega + \sigma*\omega)$$
whereas its \emph{odd part} of $\omega$ is 
$$\omega^{\odd}=\tfrac12 (\omega - \sigma*\omega).$$
All these definitions rely on a chart at infinity, but from the asymptotic rigidity of charts at infinity given by Theorem \ref{theo:asymptotic-rigidity-of isometries} and the fact that the Euclidean metric in any Euclidean chart is clearly even, it makes sense on an arbitrary asymptotically flat manifold to consider forms that are \emph{even or odd at top order}, as well as to look at \emph{the even or the odd part at top order} of a tensor or double form (note however that, due to the lowest order terms in the change of charts, a tensor which is even or odd in some chart at infinity is only even or odd at top order in any other chart).

As we have seen in the previous section, the proof of existence of the asymptotic invariants as well as their geometric invariance relies on computations mainly dealing with quantities of the form $\langle *\omega,\nu_r\rangle$ integrated over larger and larger spheres in a chart at infinity, where either $\omega\in\DD^{n-1,n}$ or $\omega\in\DD^{n,n-1}$ and integration is done with reference to the standard spherical measures and the Euclidean outer unit normals $\nu_r$ to the spheres (all these induced by the embedding of the spheres in Euclidean space). As any such $\omega$ is always of the form $d\!\vol^b\otimes\,\varphi$ or $\varphi\otimes d\!\vol^b$, the integral of the function $\langle *\omega,\nu\rangle$ is thus the same as the integral of the $(n-1)$-form $\varphi$ over the spheres. Due to the natural parity/imparity of the Euclidean volume form in a chart at infinity,
$$\sigma^*d\!\vol^b = (-1)^nd\!\vol^b,$$ 
the following result is thus clear.

\begin{prop}\label{prop:integrable-even-part}
If $D_r$ are the domains in $M$ bounded by coordinate spheres $S_r$, then for any $\varpi$ in $\DD^{n,n}$, 
$$  \lim_{r\to\infty}  \int_{D_r} *\varpi \ d\!\vol^b $$
converges if the even part of $\varpi$ is $O(|x|^{-\alpha})$ for some $\alpha>n$.
Similarly, for a double form $\omega =\varphi\otimes \widetilde{d\vol^b}$ in $\DD^{n-1,n}$,
$$ \lim_{r\to\infty}  \int_{S_r} \langle *\omega,\nu\rangle \, d\!\vol^b_{S_r} = \lim_{r\to\infty} \int_{S_r} \varphi = 0$$ 
if the even part of $\omega$ is $O(|x|^{-\alpha})$ for some $\alpha>n-1$. A similar result of course holds for double forms in $\DD^{n,n-1}$.
In both cases, the odd part contributes to zero. 
\end{prop}

We shall use this Proposition many times in what follows, and we shall often forget to mention that everything holds at top order only when lower order terms will not contribute to the computations. This is completely reminiscent of what has been already done in the proof of Theorem \ref{th:existence-GBCmass} given in Section \ref{sec:geo-invariance-GBC-mass} where we repeatedly got rid of lower order terms which did not contribute in the limits of integrals at infinity.

We now repeat the definition of the Regge-Teitelboim conditions in this setting.

\begin{defi}\label{defi:RTagain}
An asymptotically flat manifold $(M,g)$ satisfies the Regge-Teitelboim conditions of order $\tau>0$  if it is ($C^\ell$-regular) asymptotically flat of order $\tau$ and for the same chart at infinity, and, writing $g=b+e$ where $b$ is the Euclidean metric of the chart, if
$$ \left|e^{\odd} \right| + r \left|\left(\nabla^be\right)^{\odd}\right| + r^2  \left|\left((\nabla^b)^2e\right)^{\odd}\right|
+ \cdots + r^\ell \left|\left((\nabla^b)^\ell e\right)^{\odd} \right| = O\left(r^{-\tau-1}\right), 
$$
For sake of simplicity, such a metric will be said to be \emph{$\ell$-RT-asymptotically flat}.
\end{defi}
 
From the proof, it is again clear that the subsequent version of the asymptotic rigidity of isometries holds true. We can then state our main result.
 
\begin{theo}\label{th:existence-GBCcenter}
Let $k\in\NM^*$ and $(M,g)$ be a $3$-RT asymptotically flat manifold of dimension $n>2k$ and order $\tau > \tfrac{n-2k}{k+1}$ which satisfies the following conditions: $|x|\, L_k$ is in $L^1$
and the Gauss-Bonnet-Chern mass $m_k(g)$ does not vanish. Writing $g=b+e$, the \emph{Gauss-Bonnet-Chern center of mass} $C_k(g)$ is a point in $\RM^n$ whose coordinates are
$$ C_k^i = \frac{c_{n,k}}{\left( m_k(g)\right)^k}\, \lim_{r\to\infty} \int_{S_r} * \left( x^i \DR e \kn R^{\kn(k-1)} \kn b^{\kn(n-2k)} \ - \  \DR x^i\kn e \kn R^{\kn(k-1)} \kn b^{\kn(n-2k)}\right) (\nu_r) \ d\!\vol_{S_r}^b .$$
It is well defined and its value does not depend of the chart at infinity, up to the action of isometries of Euclidean space.
\end{theo}

As above, the value of the dimensional constant $c_{n,k}$ must be tuned such that the computation yields the expected result for a generalized Schwarzschild metric whose center has been translated off the origin of the chart at infinity. One also may write an alternative expression:
$$ C_k^i =  \frac{c_{n,k}}{\left( m_k(g)\right)^k}\, \lim_{r\to\infty} \int_{S_r} \left\langle x^i \DR e \kn R^{\kn(k-1)} \kn b^{\kn(n-2k)} \ - \  \DR x^i\kn e \kn R^{\kn(k-1)} \kn b^{\kn(n-2k)} \, , \, \widetilde{d\!\vol^b}\right\rangle .$$
Moreover, the assumption that $|x|\, L_k$ is in $L^1$ in Theorem \ref{th:existence-GBCcenter} may be replaced by $L_k^{\textrm{odd}} = O\left(r^{-(k+1)\tau-2k}\right)$, similarly to the statement of Definition \ref{defi:center-mass-adm}.

\begin{proofof}{Theorem~\ref{th:existence-GBCcenter}}
As before, we shall begin with a proof of the existence of the center of mass in a given chart at infinity. For this, we shall use Lemma \ref{lem:Lk-est-divergence} with $V=x^i$. Since second order derivatives of this function vanish,  
\begin{equation*}
2 \, V \, L_k \  = \  (-1)^{n-1}\, \delta * \left( V \DR e \kn\! *\! P_k \ - \  \DR V\kn e \kn\! *\! P_k \right)  \ + \ \mathcal{Q}
\end{equation*}
with the same notations as in Section \ref{sec:geo-invariance-GBC-mass}. It remains to show that the remainder $\mathcal{Q}$ is integrable. 

A careful look at the proof of Lemma \ref{lem:Lk-est-divergence} shows that $\mathcal{Q}$ is a sum of  polynomial expressions (with coefficients depending only on $b$) of the following types: 
\begin{equation*}
\begin{split} 
V \otimes e \otimes \nabla^b\nabla^be \otimes \left(R^g\right)^{k-1}, \ \
& V \otimes (\nabla^be)^2 \otimes \left(R^g\right)^{k-1}, \ \
V \otimes e \otimes \nabla^be \otimes \left(R^g\right)^{k-2}\otimes \nabla^bR^g, \ \
V \otimes r_b(e) \otimes \left(R^g\right)^{k-1}, \\ 
& \ \ \ \ \ \ \ \ \ \ \ \textrm{ and } \ dV \otimes e \otimes \nabla^be\otimes \left(R^g\right)^{k-1}.
\end{split}
\end{equation*}
Forgetting for a moment about parity-dependent decay assumptions and recalling that $V=O(|x|)$ and $dV=O(1)$, computations similar to those done in the previous section show that each of these terms is $O\left(|x|^{-(k+1)\tau-2k+1}\right)$, and this is not enough to ensure convergence. However, Proposition \ref{prop:integrable-even-part} tells us that the odd parts of the integrands can be forgotten. Since $V$ and $dV$ are obviously odd, we thus need to check the decays of the odd parts of
$e \otimes \nabla^b\nabla^be \otimes \left(R^g\right)^{k-1}$, $(\nabla^be)^2 \otimes \left(R^g\right)^{k-1}$,
$e \otimes \nabla^be \otimes \left(R^g\right)^{k-1}\otimes \nabla^bR^g$,
$r_b(e) \otimes \left(R^g\right)^{k-1}$ (all of these are to be paired with $V$),
and $e \otimes \nabla^be\otimes \left(R^g\right)^{k-1}$ (to be paired with $dV$). But the odd part of a product is a product of odd and even parts \emph{with at least one odd part}. From the definition of $3$-RT asymptotic conditions, this means that at least one of the terms in each product has one more order of decay at infinity than what is obtained by forgetting about the parity. As a result, 
$$ \mathcal{Q} \ = \  O\left(|x|^{-(k+1)\tau-2k}\right),$$
and this is again integrable.

We now proceed to the proof of the independence with respect to the changes of charts at infinity. The path is the same as above: as the center of mass obviously behaves as expected under the action of linear Euclidean isométries, we have to prove that the difference of 
$$
J_2(r) \ = \ \int_{S_r} * \left( V \DR e \kn R^{\kn(k-1)} \kn b^{\kn(n-2k)} \ - \  \DR V\kn e \kn R^{\kn(k-1)} \kn b^{\kn(n-2k)}\right) (\nu_r) \ d\!\vol_{S_r}^b $$
and 
$$ 
J_1(r) \ = \ \int_{S_r} * \left( V \DR e \kn R^{\kn(k-1)} \kn b^{\kn(n-2k)} \ - \  \DR V\kn e \kn R^{\kn(k-1)} \kn b^{\kn(n-2k)}\right) (\nu_r) \ d\!\vol_{S_r}^b 
$$
tends to zero as $r$ tends to infinity (with $V=x^i$). We shall follow the same path as in the previous section, by writing $e_2-e_1 = \mathcal{L}_\zeta b + q$ and decomposing the difference of the integrands into three contributions:
\begin{enumerate}
\item those involving either $\DR\mathcal{L}_\zeta b$ (in the l.h.s.~of the integrands) or $\mathcal{L}_\zeta b$ (in the r.h.s.), the other terms being unchanged;
\item those involving $(\DR\DL +\DL\DR)\mathcal{L}_{\zeta}b$ in the variation of the curvatures, the remainder being unchanged, where we recall that $R_2 = R_1 - \tfrac{1}{4}\left(\DR\DL +\DL\DR\right)\mathcal{L}_{\zeta}b + F(R_1,\mathcal{L}_{\zeta}b) + r_{g_1}(e_2-e_1)$;
\item all other terms.
\end{enumerate}
From the computation done in section \ref{sec:geo-invariance-GBC-mass}, all terms in the second set vanish, as a consequence of the flatness of the Euclidean space. Moreover, we can also get rid of the third set of terms: indeed, the terms having the slowest decays here are either linear in $V \nabla^b e\otimes F(R,\mathcal{L}_{\zeta}b)\otimes R^{\otimes (k-2)}$, or in $V \nabla^b e\otimes r_g(e_2-e_1)\otimes R^{\otimes (k-2)}$, or in  $dV \otimes e \otimes F(R,\mathcal{L}_{\zeta}b)\otimes R^{\otimes (k-2)}$, or in $dV \otimes e\otimes r_g(e_2-e_1)\otimes R^{\otimes (k-2)}$. At first glance, all these terms are $O\left(r^{-(k+1)\tau +2k-2}\right)$, which is insufficient to ensure that it has limit at infinity. However, we can take both parity and the special algebraic structure of these terms (as highlighted in Proposition \ref{prop:variation-courbure}) into account: since $V$ and $dV$ are odd, only the odd part of the remainder contributes. Since any such odd term contains at least one odd component, the assumptions of Theorem \ref{th:existence-GBCcenter} ensure that one has a gain of one order of decay, which is enough to imply that these terms converge to zero once integrated over larger and larger spheres.  

Thus we are left with the first set of terms, as in the proof of Section \ref{sec:geo-invariance-GBC-mass}. We now study the quantity appearing in the integrand of the difference $J_2(r)-J_1(r)$, \emph{i.e.}
\begin{equation*}
\begin{split} 
V \DR\mathcal{L}_{\zeta}b \kn T  \, - \, \DR V\kn \mathcal{L}_{\zeta}b \kn T  \ 
& = \ -  \ V \DR\left( \DL\tilde{\theta} + \DR\theta\right) \kn T  \, - \, \DR V\kn \left( \DL\tilde{\theta} + \DR\theta\right) \kn T  \\
& = \ - \ V \DR\DL\tilde{\theta} \kn T  \, -\,  \DR V \kn \DR\theta \kn T  \, - \, \DR V \kn \DL\tilde{\theta} \kn T   ,
\end{split}
\end{equation*}
where we have again denoted by $T$ the double form $R^{\kn(k-1)}\kn b^{\kn(n-2k)}$.

We now transform this expression with the help of the two relations:
$$ V \DR\DL \tilde{\theta} = \DR\left(V \DL \tilde{\theta}\right) - \DR V \kn \DL \tilde{\theta} \ \textrm{ and }\,\DR\left( V \DR\theta\right) = \DR V \kn \DR\theta ,$$
which lead to
$$
 V \DR\mathcal{L}_{\zeta}b \kn T  - \DR V\kn \mathcal{L}_{\zeta}b \kn T  
\ = \   
 \DR \left( V \left( \DR\theta - \DL\tilde{\theta} \right) \right) \kn T  + 2\, \DR V \kn \DL \tilde{\theta} \kn T  .
$$
Using the fact that $\DL\DR V =0$, the second term in the r.h.s.~can be modified as follows:
\begin{equation*}
\begin{split} 
\DR V \kn \DL \tilde{\theta} \kn T  \ 
& = \ - \DL \left( \DR V \kn \tilde{\theta}\right) \kn T  \\
& = \ - \DL \left( \DR V \kn \tilde{\theta} \kn T \right)  + L
\end{split}
\end{equation*}
where 
$$ \lim_{r\to\infty} \int_{S_r} L = 0, $$
as a consequence of the second Bianchi identity $\DL^{g_1}R =0$ and the decay assumptions (including parity).

On the other hand, we use the commutation properties of the Bianchi maps with the exterior derivatives
$$ \mathcal{B}\left(\DR \tilde{\theta}\right) \ = \ -  \DR\left( \widetilde{\mathcal{B}}\tilde{\theta}\right) -\DL\tilde{\theta} \ = \ \DR\theta - \DL\tilde{\theta}.$$
(since $ \widetilde{\mathcal{B}}\tilde{\theta} = - \theta$). Thus, 
$$
\DR \left( V \left( \DR\theta - \DL\tilde{\theta} \right) \right) \kn T  \ = \ \DR \left( V\mathcal{B}\left(\DR \tilde{\theta}\right)  \right) \kn T \ = \ \DR \left( \mathcal{B}\left(V \DR \tilde{\theta}\right)\right)  \kn T   .
$$
Commuting again the Bianchi map and the exterior derivative ($\DR\mathcal{B}=-\mathcal{B}\DR -\DL$), one has
\begin{equation*}
\begin{split} 
\DR \left( \mathcal{B}\left(V \DR \tilde{\theta}\right)\right)  \kn T 
& = \ - \mathcal{B} \left( \DR\left(V \DR \tilde{\theta}\right)\right)  \kn T   - \DL \left( V \DR \tilde{\theta}\right) \kn T     .
\end{split}
\end{equation*}
Moreover,
$$ \mathcal{B} \left( \DR\left(V \DR \tilde{\theta}\right)\right)  \kn T   \ = \ \mathcal{B} \left( \DR\left(V \DR \tilde{\theta}\right) \kn T  \right)  \ + \ L' \ = \ L' ,$$
where 
$$ \lim_{r\to\infty} \int_{S_r} L' = 0, $$
as a consequence of the first Bianchi identity $\mathcal{B}^{g_1}R =0$, the decay assumptions (including parity), and the fact that a product of a $(0,3)$-double form with a $(n-2,n-2)$-double form necessarily vanishes. Similarly,
$$ \DL \left( V \DR \tilde{\theta}\right) \kn T  \ = \  \DL \left( V \DR \tilde{\theta} \kn T \right) \ + \ L'' $$
where 
$$ \lim_{r\to\infty} \int_{S_r} L'' = 0, $$
as a consequence of the second Bianchi identity $\DL^{g_1}R =0$ and the decay assumptions (including parity). Thus, we end up with
$$ 
V \DR\mathcal{L}_{\zeta}b \kn T  \, - \, \DR V\kn \mathcal{L}_{\zeta}b \kn T  \ = \ - \DL \left( \,\left(2\,\DR V \kn \tilde{\theta}  \, + \, V \DR \tilde{\theta}\right) \kn T \right) \ + \ L'''
$$
where 
$$ \lim_{r\to\infty} \int_{S_r} L''' = 0. $$
The consequence of the whole study is that
$$ (J_2-J_1)(r) = - \int_{S_r} \left\langle * \DL \left( \,\left(2\,\DR V \kn \tilde{\theta}  \, + \, V \DR \tilde{\theta}\right) \kn T \right), \nu_r\right\rangle d\!\vol_{S_r}^b \, + \ o(1) \ = \  \int_{S_r} d\gamma \,\ + \ o(1) \ = \  o(1)$$
where $\gamma$ is the $(n-2)$-form such that 
$$ \gamma \otimes \widetilde{d\!\vol^b} =  \left( 2\,\DR V \kn \tilde{\theta} \, + \, V \DR \tilde{\theta}\right) \kn T ,$$
and the proof is done.
\end{proofof}
 
\section{Final remarks}

We conclude the paper with an alternative expression of the higher order center of mass which is first-derivative-free. An expression of the classical ADM mass of asymptotically flat manifolds involving only curvatures was first introduced by Ashtekar and Hansen in the 70s \cite{ashtekar-hansen}. It was then rigorously justified and generalized by a number of authors, with various proofs \cite{corvino-wu,ptc-remark-positive-energy,herzlich-curvature-mass,huang-center-mass,miao-tam-evaluation}. Similar expressions for the higher order masses have been devised by Wang and Wu \cite {wang-wu-chern-magic}. 

The expression we seek uses the Lovelock tensors, which are generalizations of the Einstein tensor. These are symmetric bilinear forms that are divergence free, their trace being equal to the quantity called $L_k$ in the previous parts of the paper.

\begin{prop}
For $k\geq 1$, let 
$$ T_k  =  *\!\left( R^{\kn k}\kn g^{\kn (n-2k-1)} \right).$$ 
Then, seen as an element of $\DD^{1,1}$, 
$$ \delta T_k = 0, \quad \tilde{\delta}T_k = 0, \ \textrm{ and }  \tr T_k = c\, T_k = (n-2k)\, L_k , $$
where we recall that $c$ is the contraction map in the algebra of double forms, which reduces here to the trace.
\end{prop}

We now state our last result, which yields the first-derivative-free definition of the center of mass.

\begin{theo}\label{th:center-mass-curvature}
Let $k\in\NM^*$ and $(M,g)$ be a $3$-RT asymptotically flat manifold of dimension $n>2k$ and order $\tau > \tfrac{n-2k}{k+1}$ which satisfies the following conditions: $|x|\, L_k$ is in $L^1$, and the Gauss-Bonnet-Chern mass $m_k(g)$ does not vanish. Choose a chart at infinity, and for $\alpha = 1,\dots, n$, let $X^{(\alpha)}$ be the conformal Killing vector field given in this chart by 
$$  X^{(\alpha)} \ = \ r^2\frac{\partial}{\partial x^{\alpha}} - 2 x^{\alpha}r \frac{\partial}{\partial r}$$
where $r=|x|$. Then
$$ b_{n,k}  \, \left(m_k(g)\right)^k \, C^{\alpha}_k (g) \ = \ \lim_{r\to\infty} \, \int_{S_r} T_k(X^{(\alpha)},\nu_r)\, d\!\vol^b_{S_r} $$
where $b_{n,k}$ is a constant depending only on the dimension $n$ and the integer $k$.
\end{theo}

The proof of Theorem \ref{th:center-mass-curvature} can be obtained by following closely the arguments of \cite{herzlich-curvature-mass}. It is thus left to the reader. 

\bigskip

\bibliographystyle{smfplain}

\end{document}